\documentclass[12pt,a4paper]{amsart}
\usepackage[T1]{fontenc}
\usepackage{lmodern,amsmath,amssymb,mathtools}
\usepackage[expansion=false]{microtype}
\usepackage[margin=26mm]{geometry}
\usepackage[hidelinks]{hyperref}
\hypersetup{pdftitle={A positivity conjecture of Prellberg and Stanton},pdfauthor={Thomas Prellberg},pdfsubject={Positivity of Gaussian polynomials and orbit polynomials}}
\newtheorem{theorem}{Theorem}
\newtheorem{lemma}{Lemma}
\newtheorem{corollary}{Corollary}
\newtheorem{problem}{Problem}
\newcommand{\N}{\mathbb N}
\newcommand{\Z}{\mathbb Z}
\newcommand{\Q}{\mathbb Q}
\newcommand{\qb}[3][q]{\genfrac{[}{]}{0pt}{}{#2}{#3}_{#1}}
\newcommand{\Bcal}{\mathcal B}
\newcommand{\pos}{\mathrel{\ge_q}}
\DeclareMathOperator{\lcm}{lcm}
\title{A positivity conjecture of Prellberg and Stanton}
\author{Thomas Prellberg}
\address{School of Mathematical Sciences, Queen Mary University of London, London E1 4NS, United Kingdom}
\email{t.prellberg@qmul.ac.uk}
\subjclass[2020]{Primary 05A17; Secondary 05A30, 13A50}
\keywords{Gaussian polynomial, integer partition, semi-invariant, Singer cycle}
\date{\today}
\begin{document}
\begin{abstract}
We prove the conjecture of Prellberg and Stanton (2003) that
$\frac{1-q}{(q^n;q)_m}\qb nm$ has nonnegative power-series coefficients for $0<m<n$,
except when $n$ is even and $m\in\{2,n-2\}$.
We obtain a polynomial version by bounding a M\"obius correction by
first differences of Gaussian coefficients. The proof uses partitions
and quadratic and cubic semi-invariants of binary forms.
An extension of this bound proves the Reiner--Stanton--White conjecture
on positivity of orbit polynomials for flags of subspaces.
\end{abstract}
\maketitle

\section{Introduction}
The purpose of this paper is to prove a conjecture of Prellberg and
Stanton~\cite[Conjecture~1]{PS} on the monotonicity of partition counts.
The same method proves the conjecture of Reiner, Stanton and
White~\cite[Conjecture~10.3]{RSW} on orbit polynomials for flags of subspaces.

We use the notation $[r]_q=(1-q^r)/(1-q)$,
$(x;q)_r=\prod_{j=0}^{r-1}(1-xq^j)$, and
$\qb nr=(q;q)_n/((q;q)_r(q;q)_{n-r})$.
Throughout, $\N=\{0,1,2,\ldots\}$ and $P\pos Q$ means that $P-Q$ has
nonnegative coefficients.

Prellberg and Stanton~\cite{PS} proved that, for odd $n\ge3$, the number
of partitions of $N$ with parts in $\{n,n+1,\ldots,2n-1\}$ is
nondecreasing for $N\ge1$. This continued the study of partition
monotonicity in Friedman, Joichi and Stanton~\cite{FJS}.
Prellberg and Stanton also conjectured that
\[
 F_{n,m}(q)=\frac{1-q}{(q^n;q)_m}\qb nm
\]
has nonnegative coefficients for $0<m<n$, except when $n$ is even and
$m\in\{2,n-2\}$.

Here is the partition interpretation. Let $a_{n,m}(N)$ count partitions
of $N$ with parts in $\{1,\ldots,n-m\}\cup\{n,\ldots,n+m-1\}$, using
at most $m$ parts from the first interval. The small parts fit in an
$m\times(n-m)$ rectangle; the large parts have unrestricted multiplicities.
Thus
\[
 \sum_{N\ge0}a_{n,m}(N)q^N=\frac{\qb nm}{(q^n;q)_m},\qquad
 [q^N]F_{n,m}=a_{n,m}(N)-a_{n,m}(N-1),
\]
where $a_{n,m}(-1)=0$. The conjecture says that these partition counts
are nondecreasing. We prove the following polynomial version.

\begin{theorem}\label{thm}
For $0<m<n$, except when $n$ is even and $m\in\{2,n-2\}$, the polynomial
\[
 R_{n,m}(q)=\frac{[m]_{q^{n+1}}}{[n]_q}\qb nm
 =\qb{n-1}{m-1}\frac{[m]_{q^{n+1}}}{[m]_q}
\]
has nonnegative integer coefficients. Consequently $F_{n,m}\in\N[[q]]$ for all these pairs.
\end{theorem}
For $m=1$, we have $F_{n,1}=1$. For $m\ge2$, the consequence follows from
\begin{equation}\label{RF}
 F_{n,m}=\frac{R_{n,m}}{(1-q^{m(n+1)})(q^{n+2};q)_{m-2}}
 \qquad(m\ge2).
\end{equation}
At $q=1$,
\[
 R_{n,m}(1)=\frac mn\binom nm=\binom{n-1}{m-1},
\]
the number of compositions of $n$ into $m$ positive parts. We do not yet
have a natural statistic on these compositions, or a direct partition
model, whose generating function is $R_{n,m}$ in general.
Section~\ref{furtherremarks} gives a Dyck-path construction in the coprime
case and asks for a construction in the remaining cases. Such a model,
together with~\eqref{RF}, would also count the first differences of
$a_{n,m}(N)$ directly.

The coprime case of Theorem~\ref{thm} follows from the positivity of
$\qb nm/[n]_q$, the rational $q$-Catalan polynomial associated with
the coprime pair $(m,n-m)$; see
Haiman~\cite[Propositions~2.5.3 and~2.5.4]{Haiman}
and Andrews~\cite{Andrews}.
More generally, if $g=\gcd(n,m)$, then
$[g]_q\qb nm/[n]_q\in\N[q]$ by Brunetti and Del Lungo~\cite{BrunettiDelLungo};
see also Chen and Hou~\cite[Theorem~8]{ChenHou} and
Guo and Krattenthaler~\cite[Lemma~5.2]{GuoKrattenthaler}.
To obtain $R_{n,m}$ from this quotient, one multiplies by
$[m]_{q^{n+1}}/[g]_q$. For $g>1$ this is a polynomial with coefficient
$-1$ at $q$, so the known positivity does not suffice.

We will subtract a M\"obius correction from a Gaussian polynomial and
bound that correction by the Gaussian first differences.
Sylvester's theorem identifies these differences with dimensions of
spaces of semi-invariants. We use quadratic and cubic semi-invariants
with the bidegrees of the independent family constructed by Jafari and
Najafi Amin~\cite{JafariNajafiAmin}.
For related comparisons of Gaussian polynomials, see
Chen and Jia~\cite{ChenJiaReiner,ChenJiaSylvester}; strict unimodality and
quantitative lower bounds were studied by Pak and
Panova~\cite{PakPanovaStrict,PakPanovaBounds}.

Put $a=n-m$, $g=\gcd(a,m)$, $C=n+1$, and
\begin{equation}\label{AH}
 Q_d=\qb[q^d]{n/d-1}{m/d-1}\ (d\mid g),\quad G=Q_1,\quad
 A=-\sum_{\substack{d\mid g\\d>1}}\mu(d)Q_d,\quad
 H=\frac{G-[C]_qA}{[m]_q},
\end{equation}
where $\mu$ is the M\"obius function. Empty sums are zero. We suppress
$n,m$ in $A,H$ and $Q_d$ when there is no ambiguity.

\begin{theorem}\label{correction}
For every $0<m<n$, one has $A\in\N[q]$, $H\in\Z[q]$, and
$H_{n,m}=H_{n,n-m}$. Moreover,
\[
 H\in\N[q]\quad\Longleftrightarrow\quad
 n\text{ is odd or }m\notin\{2,n-2\}.
\]
For even $n\ge4$, the exceptional values are
$H_{n,2}=H_{n,n-2}=-q^{n-1}$.
\end{theorem}
Theorem~\ref{thm} follows from Theorem~\ref{correction} and the identity
\begin{equation}\label{Rdecomp}
 R_{n,m}=[m]_{q^C}H+[C]_{q^m}A.
\end{equation}
Section~\ref{comparisons} bounds $A$ by first differences of $G$.
Section~\ref{correctionproof} uses this bound to prove
Theorem~\ref{correction}.
Corollary~\ref{singer} gives the subspace case of the orbit-polynomial
conjecture, and Section~\ref{flags} proves the result for flags.
Armstrong~\cite[Section~11.2]{Armstrong} discusses the subspace case and
proposes a further coefficientwise monotonicity conjecture.

\section{Two coefficient comparisons}\label{comparisons}
The proof of Theorem~\ref{correction} requires a lower bound for the first
differences of a Gaussian polynomial and an upper bound for the correction
$A$. We obtain both bounds from the same family of partitions.

\begin{lemma}\label{slopes}
For positive integers $a,b$, let $\Bcal_{a,b}(q)$ count partitions with parts $2,\ldots,b$, assigning cost two to each even part and cost three to each odd part, with total cost at most $a$. Then
\begin{equation}\label{slopebound}
 [q^s](1-q)\qb{a+b}{b}\ \ge\ [q^s]\Bcal_{a,b}(q)
 \qquad(0\le s\le ab/2).
\end{equation}
\end{lemma}
\begin{proof}
We use Sylvester's theorem to count semi-invariants. Give $x_i$ degree one
and weight $i$ in $\Q[x_0,\ldots,x_b]$, and let $V_{a,s}$ be the component
of degree $a$ and weight $s$. For
$D=\sum_{i=1}^b x_{i-1}\partial_{x_i}$, the theorem
(see~\cite[Theorem~3.1]{ChenJiaSylvester}, after rescaling its coordinates by $x_i=a_i/i!$) gives
\[
 \dim(\ker D\cap V_{a,s})=[q^s](1-q)\qb{a+b}{b}
 \qquad(0\le s\le ab/2).
\]
It is enough to construct independent semi-invariants indexed by the
partitions counted by $\Bcal_{a,b}$. The required degrees are two and three,
as in~\cite{JafariNajafiAmin}. Put
\[
 \Phi(y)=\sum_{i=0}^b x_i y^i,\qquad
 \Phi_2(y)=\Phi(y)\Phi(-y),\qquad
 \Phi_3(y)=\Phi(y)^2\Phi(-2y).
\]
Since $D\Phi(y)\equiv y\Phi(y)\pmod{y^{b+1}}$, the product rule gives
$D\Phi_2(y)\equiv D\Phi_3(y)\equiv0\pmod{y^{b+1}}$:
the argument multipliers sum to zero in each product. Define
\[
 f_k=\begin{cases}
 [y^k]\Phi_2(y),&k\text{ even},\\
 [y^k]\Phi_3(y),&k\text{ odd},
 \end{cases}
 \qquad 2\le k\le b.
\]
Thus $f_k$ is a semi-invariant of weight $k$ and of degree two or three
according as $k$ is even or odd. These are the assigned costs.
The highest-index term of $f_k$ is $2x_0x_k$ for even $k$ and
$(2-2^k)x_0^2x_k$ for odd $k$, with all remaining terms involving only
$x_0,\ldots,x_{k-1}$.
Successively solving for $x_k$ after inverting $x_0$ gives
\[
 \Q[x_0^{\pm1},x_1,f_2,\ldots,f_b]
 =\Q[x_0^{\pm1},x_1,\ldots,x_b].
\]
This ring has transcendence degree $b+1$, so the $b+1$ elements
$x_0,x_1,f_2,\ldots,f_b$ are algebraically independent.
A partition of weight $s$, with multiplicities $r_k$ and cost $c\le a$,
now gives the monomial $x_0^{a-c}\prod_k f_k^{r_k}$ in $\ker D\cap V_{a,s}$.
These monomials are linearly independent, proving~\eqref{slopebound}.
\end{proof}

\begin{lemma}\label{domination}
For every $0<m<n$, the correction in~\eqref{AH} satisfies
\begin{equation}\label{union}
 0\le_q A\le_q\sum_{p\mid g}Q_p,\qquad A-Q_p\pos0\quad(p\mid g),
\end{equation}
where $p$ ranges over prime divisors of $g$.
If $a=n-m\ge m\ge3$, then
\begin{equation}\label{dom}
 0\le_q A\le_q\Bcal_{a,m-1}.
\end{equation}
\end{lemma}
\begin{proof}
First we realize $A$ as the generating function of a union. Let
$B_0=\{1,\ldots,m-1\}$ and give an $(m-1)$-subset $S$ of
$\{1,\ldots,n-1\}$ weight $\operatorname{wt}(S)=\sum_{t\in S}t-\binom m2$.
Put $X_d=\{S:S\mathbin\triangle B_0\subseteq d\Z\}$. A subset in $X_d$
agrees with $B_0$ at the nonmultiples of $d$ and chooses $m/d-1$ multiples. On writing
$U=\{t/d:t\in S,\ d\mid t\}$, its weight is
$d\bigl(\sum_{u\in U}u-\binom{m/d}{2}\bigr)$, so $X_d$ has generating function $Q_d$.
Since $X_d\cap X_{d'}=X_{\lcm(d,d')}$, inclusion--exclusion shows that $A$ counts
$\bigcup_{p\mid g}X_p$. This proves~\eqref{union}.
To prove~\eqref{dom}, assume $a\ge m\ge3$. The case $g=1$ is immediate,
since $A=0$. For $g>1$ we compare coefficients at weight $s$ as follows;
here $p,p^\prime$ are primes:
\[
\begin{array}{c|l}
 s\le a & \text{inclusion--exclusion without a length constraint},\\
 s>a,\ g=p^r & \text{the direct inclusion }Q_p\le_q\Bcal_{a,m-1},\\
 s>a,\ p,p^\prime\mid g\ (p\ne p^\prime)
   & \text{disjoint classes with specified odd parts}.
\end{array}
\]
\smallskip\noindent\emph{Low weights and prime powers.}
For $d>1$, the conjugate rectangle interpretation of $Q_d$ counts partitions
all of whose parts lie in $\{d,2d,\ldots,m-d\}$, with at most $a/d$ parts.
At weight $s\le a$, the length constraint is vacuous. Also
$\Bcal_{a,m-1}$ counts every partition into $2,\ldots,m-1$, since cost
never exceeds weight. The divisibility classes intersect by least common
multiples, so their union proves~\eqref{dom} in this range.
If $g$ is a prime power, $A=Q_p$:
its partitions have at most $a/p$ parts, each of cost at most $p$.
Thus every such partition is counted by $\Bcal_{a,m-1}$, at every weight.

\smallskip\noindent\emph{Disjoint classes at higher weights.}
Suppose now that $g$ has at least two distinct prime factors and $s>a$.
We assign to each prime in~\eqref{union} a disjoint class of partitions
counted by $\Bcal_{a,m-1}$. Each class will dominate its corresponding $Q_p$.
If $2\mid g$, all-even partitions count $Q_2$ exactly.
For each odd $p\mid g$, put $h=m/p-1$. If $h\le1$, then $\deg Q_p=ha\le a$,
so there is nothing to prove at weight $s$.
For $h\ge2$, use partitions with exactly one odd part $p$ when $s$ is odd,
and exactly two odd parts $3,p$ when $s$ is even. For $p=3$ the latter means
two copies of $3$. The remaining parts are even. Thus different primes use
disjoint classes, also disjoint from the all-even class when it is present.

Let $\varepsilon=1$ or $2$ according as $s$ is odd or even, and put
\[
 E=\left\lfloor\frac{m-1}{2}\right\rfloor,\quad
 \ell_\varepsilon=\left\lfloor\frac{a-3\varepsilon}{2}\right\rfloor,
 \quad w_1=p,\quad w_2=p+3.
\]
After removing these odd parts, we may use at most $\ell_\varepsilon$ even
parts, each at most $2E$. Halving the even parts gives the target count
$[z^v]\qb[z]{E+\ell_\varepsilon}{E}$, where $v=(s-w_\varepsilon)/2$.
A nonzero source coefficient has $p\mid s$ and $s\le ha$. By Gaussian
symmetry it equals $[z^u]\qb[z]{h+a/p}{h}$, where
$u=\min(s,ha-s)/p$.
The comparison has a simple diagram interpretation. Divide the original
parts by $p$ and conjugate to obtain a partition in the $h\times(a/p)$
rectangle. If $s>ha/2$, take its rectangular
complement, so its size is $u$. After transposing this rectangle if needed,
view the same Ferrers diagram inside the target rectangle. This injects the
source partitions into the target partitions of size $u$.
Symmetry and unimodality compare the target coefficients at $u$ and $v$,
provided
\begin{equation}\label{fit}
 h\times(a/p)\text{ fits in }E\times\ell_\varepsilon,\qquad
 0\le u\le v\le E\ell_\varepsilon-u,
\end{equation}
where transposition of the source rectangle is allowed. Gaussian unimodality
follows from the same Sylvester theorem used in Lemma~\ref{slopes}.

\smallskip\noindent\emph{Verification of the two bounds.}
Since $h\ge2$, we have $a\ge m\ge3p$, and hence $s>3p$.
For $u\le v$, it suffices that $(p-2)s\ge pw_\varepsilon$.
When $s$ is odd, use $s>3p$ and $3(p-2)\ge p$.
When $s$ is even and $p\ge5$, use $s>3p$ and $3(p-2)\ge p+3$.
For $p=3$ and even $s$, the hypotheses give $m\ge12$ and $6\mid s$;
therefore $s>a\ge m$ implies $s\ge18$, as required.
Also $u+v\le(ha-w_\varepsilon)/2$. Thus the upper bound in~\eqref{fit}
follows if
\begin{equation}\label{eta}
 \eta_\varepsilon:=2E\ell_\varepsilon-ha+w_\varepsilon\ge0.
\end{equation}
We check the rectangle inclusion and~\eqref{eta} together. It suffices to
treat $\varepsilon=2$: $\ell_1\ge\ell_2+1$ and
$\eta_1\ge\eta_2+2E-3$, with $2E\ge3$. Indeed $m\ge12$ for $p=3$
(the value $m=9$ would make $g$ a prime power), and $m\ge15$ for $p\ge5$.

If $p\ge5$, then $a\ge m\ge15$, $h\le E$, $\ell_2\ge(a-7)/2\ge a/p$, and
\[
 2E\ell_2-ha\ge\frac{(m-2)(a-7)}2-(m/5-1)a
 =\frac{m(3a-35)}{10}+7>0.
\]
For $p=3$, write $m=3r$, $a=3t$, so $t\ge r\ge4$. Both $E$ and $\ell_2$ are at least $r-1$. If $t\ge6$, then $\ell_2\ge t$; otherwise $E\ge5\ge t$. Thus the rectangles fit. As $a$ increases by three, $\ell_2$ increases by at least one and $2E>3h$, so $\eta_2$ increases. At $a=m$ its value is
\[
 \eta_2=\begin{cases}
 (m-6)(m-12)/6,&m\text{ even},\\
 ((m-9)^2-24)/6,&m\text{ odd}.
 \end{cases}
\]
These are nonnegative, since $m\ge12$ and odd $m\ge15$.
At the smallest case $m=a=12$, the source rectangle is $3\times4$ and
the target is $5\times3$, illustrating why transposition is allowed in~\eqref{fit}.
Thus~\eqref{fit} holds in every case. The disjoint classes dominate the
prime sum in~\eqref{union} for $s>a$, and the low weights were already proved.
\end{proof}

\section{Positivity of the corrected quotient}\label{correctionproof}
The following division lemma turns symmetric unimodality into positivity
of a quotient.

\begin{lemma}\label{division}
Let $r,\tau$ be positive integers and let $P\in\N[q]$ be symmetric and
unimodal. If $P/[r]_{q^\tau}$ is a polynomial, it belongs to $\N[q]$.
\end{lemma}
\begin{proof}
This is~\cite[Lemma~5.1]{GuoKrattenthaler}, with exponents $\tau$ and $r\tau$.
We include the short argument. Write $P=\sum_{j=0}^D p_jq^j$ with $p_j=p_{D-j}$,
allowing zero coefficients at both ends, and set $p_j=0$ outside $[0,D]$.
Below the midpoint of the quotient its coefficient at $k$ is
\[
 \sum_{t\ge0}(p_{k-tr\tau}-p_{k-\tau-tr\tau})\ge0.
\]
All indices lie at or below $(D-(r-1)\tau)/2\le D/2$, where the
coefficients of $P$ are nondecreasing. Symmetry gives nonnegativity for the
remaining coefficients. Division by the monic polynomial $[r]_{q^\tau}$
gives integrality. This also covers a specified centre with zero
coefficients at both ends, and the case $P=0$.
\end{proof}

\begin{proof}[Proof of Theorem~\ref{correction}]
Lemma~\ref{domination} gives $A\in\N[q]$ for all $n>m$.
For $H$, we first prove polynomiality and remove the small widths.
We then write $H$ as a sum of two nonnegative polynomials.

\smallskip\noindent\emph{1. Polynomiality.}
Let $\zeta$ have order $e\mid m$, with $e>1$. For $d\mid g$, the order of
$\zeta^d$ is $e'=e/\gcd(e,d)$, and $e'\mid m/d$.
If $e\nmid n$, then $e'\nmid n/d$. In the Gaussian $Q_d(\zeta)$, the
bottom residue is $e'-1$ and the top residue is smaller.
The $q$-Lucas evaluation~\cite[Theorem~2.2, pp.~131--132]{Sagan}
(see also the earlier proofs~\cite{Olive,Desarmenien}) therefore gives
$Q_d(\zeta)=0$, including $G(\zeta)=0$ when $d=1$.
If $e\mid n$, the same evaluation gives
\[
 G(\zeta)=f(e),\qquad Q_d(\zeta)=f(\lcm(d,e)),\qquad
 f(t)=\binom{n/t-1}{m/t-1}.
\]
Since $e\mid g$, choose a prime $p\mid e$. Pair squarefree divisors of $g$
by inserting or deleting $p$. The pairs have the same $\lcm(d,e)$ and
opposite M\"obius signs. Hence $\sum_{d\mid g}\mu(d)Q_d(\zeta)=0$, so
$A(\zeta)=G(\zeta)$. Also $[C]_\zeta=1$.
In both cases the numerator of $H$ vanishes at $\zeta$.
Since $[m]_q$ is monic with simple roots, $H\in\Z[q]$.

\smallskip\noindent\emph{2. Symmetry and small widths.}
For each $d\mid g$,
\begin{equation}\label{sym}
 \frac{Q_d}{[m]_q}=\frac{\qb[q^d]{n/d}{m/d}}{[n]_q}.
\end{equation}
Gaussian symmetry therefore gives $H_{n,m}=H_{n,n-m}$.
We may therefore assume $a=n-m\ge m$ when proving positivity.
If $m=1$, then $H=1$; if $m=2$ and $n$ is odd, then
$H=[(n-1)/2]_{q^2}$. For even $n\ge4$, one has $G=[n-1]_q$ and $A=1$, so
\[
 H_{n,2}=\frac{[n-1]_q-[n+1]_q}{[2]_q}=-q^{n-1}.
\]
Symmetry gives the same value at $m=n-2$, so these exceptions are
necessary. It remains to prove positivity for $a\ge m\ge3$.

Put $T=(m-1)a$, $L=T+1-C=(m-2)a-m$, and
$\delta_s=[q^s](1-q)G$. Lemmas~\ref{slopes} and~\ref{domination} give
\[
 \delta_s\ge A_s:=[q^s]A\qquad(0\le s\le T/2).
\]
We will subtract intervals from the usual unimodal expansion of $G$,
using these inequalities to keep the coefficients nonnegative. The
subtraction leaves a reflected correction, which we treat first.

\smallskip\noindent\emph{3. The reflected correction.}
We claim that
\begin{equation}\label{W}
 W=\frac{A(q)-q^LA(q^{-1})}{1-q^m}\in\N[q].
\end{equation}
For $d\mid g$, $d>1$, let $D_d=(m/d-1)a=\deg Q_d$ and
$b_d=L-D_d=(m-1-m/d)a-m$.
Then $b_d\ge a-m\ge0$ and $g\mid b_d$.
A root of $[m/g]_{q^g}$ has order dividing $m$ but not $g$, and therefore
not $n$. Step~1 gives $[m/g]_{q^g}\mid Q_d$.
Both polynomials are in $t=q^d$, with divisor $[m/g]_{t^{g/d}}$.
Since $Q_d$ is symmetric and unimodal in $t$, Lemma~\ref{division} in
that variable gives
\[
 K_d=\frac{(1-q^{b_d})Q_d}{1-q^m}
 =[b_d/g]_{q^g}\frac{Q_d}{[m/g]_{q^g}}\in\N[q],
\]
where $[0]_{q^g}=0$.
Reciprocity of $Q_d$ now gives
$W=-\sum_{d\mid g,\,d>1}\mu(d)K_d\in\Z[q]$.

If $g=1$, then $W=0$. Otherwise let $p$ be its least prime divisor and
$V=A-Q_p\pos0$. If $g$ is a prime power, then $V=0$ and $W=K_p$.
If the next prime divisor is $p'\ge3$, the union interpretation gives
\[
 \deg V\le(m/p'-1)a\le L/2,
\]
since
\[
 \frac L2-(m/p'-1)a
 =\frac m2\bigl((1-2/p')a-1\bigr)\ge0.
\]
In either case,
\[
 W-K_p=\frac{1}{[m]_q}\sum_{0\le j\le L/2}[q^j]V\,q^j[L-2j]_q.
\]
The nonzero intervals in the numerator all have centre $(L-1)/2$.
Their sum is nonnegative, symmetric and unimodal. Since $W-K_p$ is a
polynomial, Lemma~\ref{division} gives $W-K_p\pos0$ and proves~\eqref{W}.

\smallskip\noindent\emph{4. Symmetric unimodal completion.}
We now use $\delta_s\ge A_s$ to form
\begin{equation}\label{P}
 \begin{split}
 P&=\sum_{0\le s\le T/2}(\delta_s-A_s)q^s[T+1-2s]_q\\
  &=G-\frac{A(q)-q^{T+1}A(q^{-1})}{1-q}.
 \end{split}
\end{equation}
The second equality follows from the first-difference expansion of $G$
and the bound
$\deg A\le(m/2-1)a<T/2$ from~\eqref{union}; the equality also holds if $A=0$.
The first expression is a nonnegative sum of intervals with centre $T/2$.
Thus $P$ is nonnegative, symmetric and unimodal.
Since $T+1=C+L$, direct cancellation yields
\begin{equation}\label{completion}
H=\frac{P}{[m]_q}+q^CW.
\end{equation}
Since $H,W\in\Z[q]$, the quotient $P/[m]_q$ is a polynomial.
It is nonnegative by Lemma~\ref{division}, and $W$ is nonnegative by
step~3. Hence $H\in\N[q]$.
\end{proof}

\begin{proof}[Proof of Theorem~\ref{thm} and the exceptional cases]
Theorem~\ref{correction} and~\eqref{Rdecomp} prove $R_{n,m}\in\N[q]$ for
the asserted pairs, and~\eqref{RF} gives $F_{n,m}\in\N[[q]]$ when $m\ge2$.
For $m=1$ both expressions equal $1$.

For even $n\ge4$, the excluded width two has
\[
 R_{n,2}=[n]_{q^2}-q^{n-1},\qquad
 F_{n,2}=\frac{1-q^{n-1}}{(1-q^2)(1-q^{n+1})}.
\]
In each expression the coefficient of $q^{n-1}$ is $-1$.
Gaussian symmetry gives
\[
 F_{n,n-2}=\frac{F_{n,2}}{(q^{n+2};q)_{n-4}},\qquad
 R_{n,n-2}=R_{n,2}[(n-2)/2]_{q^{2(n+1)}}.
\]
The additional factors have constant term $1$ and no other term of degree
at most $n-1$, so the same negative coefficient persists. Thus all the
exceptions in Theorem~\ref{thm} are necessary.
\end{proof}

Dennis Stanton observed that Theorem~\ref{correction} implies the following
corollary, and suggested the flag extension in Section~\ref{flags}.
The corollary proves the free-orbit assertion for subspaces in Reiner,
Stanton and White~\cite[Conjecture~10.3]{RSW}: the case $a=1$ and
$\mathbf k=(m,n-m)$. For prime powers $q$, the quotient below counts orbits
of size $[n]_q$ of the Singer cycle
$\mathbb F_{q^n}^{\times}/\mathbb F_q^{\times}$ on the $m$-dimensional
$\mathbb F_q$-subspaces of $\mathbb F_{q^n}$; its polynomiality was proved
in~\cite[Corollary~10.2]{RSW}.

\begin{corollary}\label{singer}
For $0<m<n$,
\[
 \frac1{[n]_q}\sum_{d\mid\gcd(n,m)}
 \mu(d)\qb[q^d]{n/d}{m/d}\in\N[q].
\]
\end{corollary}
\begin{proof}
Put
\[
 B=-\sum_{\substack{d\mid g\\d>1}}\mu(d)\qb[q^d]{n/d}{m/d}.
\]
We again use the subset construction of Lemma~\ref{domination}, now with
$m$-subsets $S$ of $\{1,\ldots,n\}$ of weight
$\sum_{t\in S}t-\binom{m+1}{2}$. The sets
\[
 Y_d=\{S:S\mathbin\triangle\{1,\ldots,m\}\subseteq d\Z\}
 \qquad(d\mid g)
\]
have generating polynomials $\qb[q^d]{n/d}{m/d}$ and satisfy
$Y_d\cap Y_e=Y_{\lcm(d,e)}$. Inclusion--exclusion therefore expresses
$B$ as the generating polynomial of $\bigcup_{p\mid g}Y_p$, where $p$
runs over the prime divisors of $g$. In particular, $B\in\N[q]$.

Identity~\eqref{sym} gives $B=[n]_qA/[m]_q$, and hence
\[
 \frac{\qb nm-B}{[n]_q}=H+qB,
\]
since $[n+1]_q=1+q[n]_q$.
Theorem~\ref{correction} proves the assertion except when $n$ is even
and $m\in\{2,n-2\}$. In these cases $B=[n/2]_{q^2}$ and $H=-q^{n-1}$,
so the quotient is $q[n/2-1]_{q^2}\in\N[q]$.
\end{proof}

\section{Orbit polynomials of flags}\label{flags}
For a composition $\mathbf k=(k_1,\ldots,k_r)$ of $N$, write
\[
 \qb{N}{k_1,\ldots,k_r}
 =\frac{(q;q)_N}{\prod_{i=1}^r(q;q)_{k_i}}.
\]
We now prove the flag case of Conjecture~10.3 of Reiner, Stanton and
White~\cite{RSW}.

\begin{theorem}\label{flagpositivity}
Let $\mathbf k$ be a composition of $N$ and let $n$ be a positive divisor
of $N$. Then
\[
 \mathcal O_{n,\mathbf k}(q)
 =\frac1{[n]_q}
 \sum_{d\mid\gcd(n,k_1,\ldots,k_r)}
 \mu(d)\qb[q^d]{N/d}{k_1/d,\ldots,k_r/d}
 \in\N[q].
\]
\end{theorem}
Here $N=an$ in the notation of~\cite{RSW}, and the quotient is the
free-orbit polynomial. For prime powers $q$, it counts the orbits of
size $[n]_q$ of
$\mathbb F_{q^n}^{\times}/\mathbb F_q^{\times}$ on flags in
$\mathbb F_{q^N}$ with successive dimension increments $k_1,\ldots,k_r$.
For $d\mid n$, the number of orbits with stabilizer
$\mathbb F_{q^d}^{\times}/\mathbb F_q^{\times}$ is
$\mathcal O_{n/d,\mathbf k/d}(q^d)$ if $d\mid\gcd(n,k_1,\ldots,k_r)$,
and zero otherwise~\cite[Proposition~9.2]{RSW}.
Their sum counts all orbits. Thus Theorem~\ref{flagpositivity} covers
every orbit size and every $a\ge1$.

The proof needs the following coefficient comparison for the full
Gaussian polynomial.

\begin{lemma}\label{fullslopes}
Let $a\ge b>0$, $g=\gcd(a,b)$, and put
\[
 T_d=\qb[q^d]{(a+b)/d}{b/d}\quad(d\mid g),\qquad
 B=-\sum_{\substack{d\mid g\\d>1}}\mu(d)T_d.
\]
Then $B\in\N[q]$ and
\[
 [q^s](1-q)T_1\ge[q^s]B\qquad(0\le s\le ab/2).
\]
\end{lemma}
\begin{proof}
The subset union in Corollary~\ref{singer} gives
$0\le_q B\le_q\sum_{p\mid g}T_p$, with $p$ prime.
By Lemma~\ref{slopes}, the first differences are bounded below by
$\Bcal_{a,b}$. Since $T_d$ counts partitions with parts $d,2d,\ldots,b$
and at most $a/d$ parts, the low-weight and prime-power arguments of
Lemma~\ref{domination} give $B\le_q\Bcal_{a,b}$ for $s\le a$, and for
all $s$ if $g$ is a prime power. The case $g=1$ is immediate.

Suppose $g$ has at least two distinct prime factors and $s>a$.
We bound the prime sum by constructing independent semi-invariants.
Put $c_j=[y^j]\Phi_3(y)$, as in Lemma~\ref{slopes}. For
$2\le j\le b$, $c_j$ is a cubic semi-invariant of weight $j$, and
$c_j=f_j$ when $j$ is odd. If $2\mid g$, the even $f_j$ give monomials counted
by $T_2$, after padding with $x_0$ to degree $a$. For each odd prime
$p\mid g$, multiply monomials in the even $f_j$ by $c_p$ for odd $s$
and by $c_{p+3}$ for even $s$, then pad to degree $a$. These markers are
available since $b\ge2p\ge p+3$.

To prove independence, use
\[
 \Phi_3(y)\Phi_3(-y)=\Phi_2(y)^2\Phi_2(2y).
\]
Since $c_0=x_0^3$ and $c_1=0$, coefficient comparison expresses each
even $c_k$ recursively in the even $f_j$ and odd $c_j$ with $j\le k-3$,
over $\Q(x_0)$. In particular,
\[
 c_6=\frac{c_3^2}{2x_0^3}+R_3,\qquad
 c_{p+3}=\frac{c_3c_p}{x_0^3}+R_p\quad(p>3),
\]
where $R_3$ involves only even generators, and $R_p$ involves even
generators and odd $c_j$ with $j<p$.
By Lemma~\ref{slopes}, the even $f_j$ and odd $c_j$ are algebraically
independent. At even weight, compare the largest prime first, then
$c_3^2$ if $3\mid g$, and finally the all-even family. At odd weight,
compare the distinct $c_p$. The even monomials within each family are
also independent. By Sylvester's formula, it remains to count them.

Each odd-prime marker has degree three, so its family contributes
\[
 [z^v]\qb[z]{E+\ell}{E},\qquad
 E=\lfloor b/2\rfloor,\quad \ell=\lfloor(a-3)/2\rfloor,
 \quad v=(s-w)/2,
\]
where $w=p$ or $p+3$ according as $s$ is odd or even.
For a nonzero coefficient of $T_p$, put $h=b/p$ and
$u=\min(s,ha-s)/p$. By the comparison in Lemma~\ref{domination},
it suffices to show that
\[
 h\times(a/p)\text{ fits in }E\times\ell,
 \qquad 0\le u\le v\le E\ell-u.
\]

Except when $(a,b)=(6,6)$, we have $(p-2)a\ge4p$: this follows from
$a\ge b\ge2p$ for $p\ge5$, while $p=3$ forces $a\ge12$.
Thus $h\le E$ and $\ell\ge(a-4)/2\ge a/p$, so the rectangles fit.
Since $s>a$ and $p\mid s$,
\[
 (p-2)\frac{s}{p}\ge(p-2)\left(\frac ap+1\right)\ge p+2.
\]
When $s$ is even the left side is even, hence at least $p+3$;
in either parity it is at least $w$. Consequently $u\le s/p\le v$.
Also $u+v\le(ha-w)/2$, so the remaining bound follows from
$\eta_w:=2E\ell-ab/p+w\ge0$. If $b$ is even, then
\[
 \eta_w=b(\ell-a/p)+w\ge w>0.
\]
If $b$ is odd, the two distinct common primes force $a\ge b\ge15$, and
\[
 \eta_w\ge\frac{(b-1)(a-4)}2-\frac{ab}{3}+3
 =\frac{(a-12)(b-3)}6-1\ge5.
\]

For $(a,b)=(6,6)$, only the count needs a separate check. Above weight
$6$, $T_3=\qb[q^3]{4}{2}$ contributes once at weights $9$ and $12$.
The same construction gives $x_0c_3f_6$ and $x_0c_6f_6$, respectively,
both of degree six and independent of the all-even family. Thus the
marked families dominate the prime sum for every $s>a$.
\end{proof}

\begin{proof}[Proof of Theorem~\ref{flagpositivity}]
By~\cite[Corollary~10.2]{RSW}, the quotient lies in $\Z[q]$.
If $r=1$, the quotient is $1$ for $n=1$ and $0$ otherwise.
Assume $r\ge2$, and put
\[
\begin{aligned}
 M&=\qb{N}{k_1,\ldots,k_r},\qquad D=\sum_{i<j}k_i k_j,\\
 g&=\gcd(n,k_1,\ldots,k_r),\qquad
 B=-\sum_{\substack{d\mid g\\d>1}}\mu(d)
       \qb[q^d]{N/d}{k_1/d,\ldots,k_r/d}.
\end{aligned}
\]
Then $\mathcal O_{n,\mathbf k}=(M-B)/[n]_q$.

First factor $M=\prod_{i=1}^{r-1}G_i$, where
$N_i=k_i+\cdots+k_r$ and $G_i=\qb{N_i}{k_i}$.
Let $B_i$ be its correction from Lemma~\ref{fullslopes}, using all
primes dividing $\gcd(N_i,k_i)$. By Gaussian symmetry we may take the
smaller rectangle side as the width.
Apply the weighted $k_i$-subset construction of Corollary~\ref{singer}
to each factor, with sets $Y_{i,d}$.
For $d\mid g$, the product $Y_d=\prod_iY_{i,d}$ has generating
polynomial $\qb[q^d]{N/d}{k_1/d,\ldots,k_r/d}$ and satisfies
$Y_d\cap Y_e=Y_{\lcm(d,e)}$.
Thus $B$ counts $\bigcup_{p\mid g}\prod_iY_{i,p}$, which is contained in
$\prod_i\bigcup_{p\mid\gcd(N_i,k_i)}Y_{i,p}$, whence
\[
 0\le_q B\le_q\prod_i B_i.
\]
Also $\deg B\le D/2$ when $B\ne0$, since each term with $d>1$ has
degree $D/d$.

For a nonnegative symmetric unimodal polynomial $F$, with
$F(q)=q^LF(q^{-1})$, put
\[
 \Delta F=\sum_{j=0}^{\lfloor L/2\rfloor}
             \bigl([q^j](1-q)F\bigr)q^j.
\]
The symmetric interval expansion and
\[
 [u]_q[v]_q=\sum_{t=0}^{\min(u,v)-1}q^t[u+v-1-2t]_q
 \qquad(u,v\ge1)
\]
give $\Delta(FG)\pos(\Delta F)(\Delta G)$ for two such polynomials
$F,G$, by retaining the interval with $t=0$ in each product.
Since $M$ has symmetry degree $D$, Lemma~\ref{fullslopes} gives
\[
 \Delta M\pos\prod_i\Delta G_i\pos\prod_iB_i\pos B.
\]

To apply Lemma~\ref{division}, write $\delta_j=[q^j](1-q)M$ and
$b_j=[q^j]B$, and put
\[
\begin{aligned}
 P&=\sum_{j=0}^{\lfloor D/2\rfloor}
       (\delta_j-b_j)q^j[D+1-2j]_q,\\
 V&=\sum_{j=0}^{\lfloor D/2\rfloor}b_jq^j[D-2j]_q.
\end{aligned}
\]
Both are nonnegative, symmetric and unimodal; $V=0$ is allowed.
The interval expansions give
\[
 V=\frac{B(q)-q^DB(q^{-1})}{1-q},\qquad M-B=P+qV.
\]
Since $(M-B)/[n]_q$ is a polynomial, $B(\zeta)=M(\zeta)$ at every
nontrivial $n$th root of unity $\zeta$. By reciprocity,
\[
 B(\zeta)=M(\zeta)=\zeta^DM(\zeta^{-1})
                         =\zeta^DB(\zeta^{-1}),
\]
so $V(\zeta)=0$. Hence $V/[n]_q\in\Z[q]$, and
$P/[n]_q=\mathcal O_{n,\mathbf k}-qV/[n]_q$ is also an integral polynomial.
Lemma~\ref{division} makes both quotients nonnegative, and
$\mathcal O_{n,\mathbf k}=P/[n]_q+qV/[n]_q$ proves the theorem.
\end{proof}

\section{Remarks and questions}\label{furtherremarks}
Theorem~\ref{thm} leaves several combinatorial questions. We first compare
its product formula with other positivity results.
Gatzweiler and Krattenthaler~\cite{GatzweilerKrattenthaler} study quotients
of Gaussian polynomials with a common top argument (Conjecture~1), and
products $\prod_i(1-q^{r+i})^{a_i}/(1-q^i)^{a_i}$ with a symmetric
exponent sequence (Conjecture~3). The latter extend Stanton's fake
Gaussian sequences~\cite{StantonFake}.
In $R_{n,m}$ the factor $[m]_{q^{n+1}}$ dilates the exponents of $[m]_q$;
in those products a common shift translates the degrees. We do not
identify $R_{n,m}$ as a specialization of either family.
For cyclotomic generating functions more generally, see
Billey and Swanson~\cite{BilleySwanson} and Gatzweiler,
Levic\'an-Santib\'a\~nez and Yoshida~\cite{GLY}.

There is a simple obstruction to applying the regular-sequence criterion
of~\cite[Proposition~3.3]{GLY} to the product
\[
 R_{n,m}=\frac{(1-q^{mC})\prod_{j=n-m+1}^{n-1}(1-q^j)}
 {(1-q^C)\prod_{j=2}^{m}(1-q^j)},\qquad C=n+1
\]
for $m\ge2$. The criterion requires every selection of denominator
degrees to generate an additive monoid containing at least as many
numerator degrees, counted with multiplicity.
Suppose $g=\gcd(n,m)>1$. Select the $m/g+1$ denominator degrees
$C,g,2g,\ldots,m$. They generate $\langle g,C\rangle$.
Below $C$, this monoid contains only multiples of $g$, so it contains
exactly $m/g-1$ degrees from the numerator interval and also the degree
$mC$. There are only $m/g$ such numerator degrees.
Thus the criterion fails whenever $g>1$, even in positive cases such as
$(n,m)=(6,3)$. This concerns the displayed product; another Hilbert-series
realization may still be possible.

At $q=1$ there is an elementary necklace interpretation.
Let $\mathcal N_{n,m}^{\rm prim}$ be the binary necklaces of length $n$
with $m$ ones and least period $n$. M\"obius inversion on repetitions,
followed by division by the orbit size $n$, gives
\[
 \#\mathcal N_{n,m}^{\rm prim}
 =\frac1n\sum_{d\mid g}\mu(d)\binom{n/d}{m/d},\qquad
 G(1)-A(1)=m\,\#\mathcal N_{n,m}^{\rm prim}.
\]
The second equality uses
$\binom{n/d-1}{m/d-1}=(m/n)\binom{n/d}{m/d}$.
Lemma~\ref{domination} interprets $G-A$ at general $q$ as the weight
enumerator of subsets outside $\bigcup_{p\mid g}X_p$.
It would be interesting to find a cyclic action and a compatible statistic
giving a cyclic-sieving interpretation of these polynomials.

A direct proof of the original conjecture would give an injection between
successive partition sets. Let $\mathcal P_{n,m}(N)$ be the set counted by
$a_{n,m}(N)$.
\begin{problem}\label{injectionproblem}
For each nonexceptional pair $0<m<n$, give a uniform combinatorial rule
for an injection from $\mathcal P_{n,m}(N)$ into
$\mathcal P_{n,m}(N+1)$ for every $N\ge0$.
\end{problem}
The polynomial version suggests a second question.
\begin{problem}\label{statisticproblem}
Find a natural nonnegative integer statistic on compositions of $n$ into
$m$ positive parts such that, for every nonexceptional pair $0<m<n$,
\[
 R_{n,m}(q)=
 \sum_{\substack{a_1+\cdots+a_m=n\\a_1,\ldots,a_m\ge1}}
 q^{\operatorname{stat}(a_1,\ldots,a_m)}.
\]
\end{problem}
Since $R_{n,m}(1)=\binom{n-1}{m-1}$, Theorem~\ref{thm} permits an
arbitrary assignment of exponents to compositions. We ask for a direct
combinatorial rule. When $m\ge2$, one could then use~\eqref{RF} to count
$a_{n,m}(N)-a_{n,m}(N-1)$ by adjoining a partition with parts in
$\{m(n+1)\}\cup\{n+2,\ldots,n+m-1\}$, the latter set being empty for $m=2$.

There is such a statistic when $n,m$ are coprime. The rational shuffle
theorem~\cite{Mellit} and the specialization in
Garsia--Leven--Wallach--Xin~\cite[identity~(I.1)]{GLWX} give
\[
 \frac{1}{[m]_q}\qb{n-1}{m-1}
 =\sum_{D\in\mathcal D_{m,n-m}}q^{\operatorname{coarea}(D)+\operatorname{dinv}(D)}.
\]
Here $\mathcal D_{m,n-m}$ consists of Dyck paths above the diagonal in the
$m\times(n-m)$ rectangle, with the statistics of~\cite{GLWX}.
Given a composition $(a_1,\ldots,a_m)$, the cycle lemma gives a unique
cyclic rotation of $(a_1-1,\ldots,a_m-1)$ encoding the numbers of north
steps before successive east steps of a path $D$. Let
$j\in\{0,\ldots,m-1\}$ count the left cyclic rotations taking this
encoding back to the original tuple. Then
\[
 \operatorname{stat}(a_1,\ldots,a_m)
 =(n+1)j+\operatorname{coarea}(D)+\operatorname{dinv}(D)
\]
has generating polynomial $R_{n,m}$. The remaining question is to find
a rule for the noncoprime pairs.

\section*{Declaration of generative AI and AI-assisted technologies in the manuscript preparation process}
During the preparation of this work, I used OpenAI ChatGPT to explore
mathematical arguments and to assist in drafting and editing parts of the
manuscript. I reviewed and revised the resulting material and take full
responsibility for the content of this article.

\end{document}